\documentclass[letterpaper, 10 pt, conference]{ieeeconf}

\usepackage[utf8]{inputenc}

\IEEEoverridecommandlockouts                              
\usepackage{graphicx,cite}

\usepackage[linesnumbered,ruled,vlined]{algorithm2e}

\usepackage{enumerate} 
\usepackage{amsmath} 

\usepackage{amsfonts}

\newlength{\bibitemsep}
\newlength{\bibparskip}
\let\oldthebibliography\thebibliography
\renewcommand\thebibliography[1]{%
  \oldthebibliography{#1}%
  \setlength{\parskip}{\bibitemsep}%
  \setlength{\itemsep}{\bibparskip}%
}

\usepackage{color}

\usepackage{hyperref}
\usepackage{xcolor}
\hypersetup{
	colorlinks,
	linkcolor={red!80!black},
	citecolor={blue!80!black},
	urlcolor={blue!80!black}
}

\usepackage[english]{babel}
\usepackage[utf8]{inputenc}
\usepackage[noend]{algpseudocode}

\newtheorem{thm}{Theorem}
\newtheorem{lem}[thm]{Lemma}
\newtheorem{prop}[thm]{Proposition}

\newtheorem{defn}{Definition}

\newtheorem{rem}[thm]{Remark}

\usepackage{balance}

\usepackage{bbm} 

\title{\LARGE \bf
Distributed Linear Quadratic Optimal Control:\\
Compute Locally and Act Globally
}

\author{Junjie~Jiao,~Harry~L.~Trentelman,~\IEEEmembership{Fellow,~IEEE,} and M.~Kanat~Camlibel,~\IEEEmembership{Member,~IEEE} 
\thanks{The authors are with the Bernoulli Institute for Mathematics, Computer
	Science and Artificial Intelligence, University of Groningen, The Netherlands 
      (Email:  { j.jiao@rug.nl; h.l.trentelman@rug.nl; 
      	m.k.camlibel@rug.nl)}}%
}

\begin{document}

\maketitle
\thispagestyle{empty}
\pagestyle{empty}

\begin{abstract}

In this paper we consider the distributed linear quadratic control problem for 
 networks of agents with single integrator dynamics.
We first establish a general formulation of the distributed LQ problem and show that the optimal control gain depends on global information on the network. Thus, the optimal protocol can only be computed in a centralized fashion.
In order to overcome this drawback, we propose the design of protocols that are computed in a decentralized way.  We will write the global cost functional  as a sum of {\em local} cost functionals, each associated with one of the agents. In order to achieve `good' performance of the controlled network, each agent then computes its own local gain, using sampled information of its neighboring agents. This decentralized computation will only lead to suboptimal global network behavior. However, we will show that the resulting network will reach consensus.
A simulation example is provided to illustrate the performance of the proposed protocol.
\end{abstract}

\section{Introduction}\label{sec_problem_formulation}
%
The distributed linear quadratic (LQ) optimal control problem is the problem of interconnecting a finite number of identical agents according to a given network graph so that consensus is achieved in an optimal way. Each agent receives input only from its neighbors, in the form of a linear feedback of the relative states amplified by a certain constant gain. Such control law is called a {\em distributed diffusive control law}. The problem of minimizing a given quadratic cost functional over all distributed diffusive control laws that achieve consensus is then called the distributed LQ problem corresponding to this cost functional. 

In the case that the agent dynamics is given by a general state space system, this optimal control problem is non-convex and difficult to solve, and it is unclear whether a solution exists in general, see \cite{Jiao2018}. In contrast, for the case of single integrator dynamics it is fairly easy to find an explicit expression for the optimal distributed diffusive control law, see, for example, \cite{wei_ren2010}. 
Although a solution to the problem is available, it turns out however that {\em global information} on the network is needed to compute this optimal control law.
More specifically, the optimal distributed diffusive control law can be computed only by a (virtual) supervisor that knows the network graph and the initial states of all the agents. Thus, although the resulting optimal control law {\em operates} in a distributed fashion, its actual computation can only be performed in a {\em centralized} way. 

Formulating the distributed LQ problem as a problem of  minimizing a {\em global} cost functional is therefore not practical. Indeed, the centralized computation requires that the local optimal gains needs to be re-designed by the supervisor in case that changes in the network occur. For example, by adding or removing agents from the network, its graph will change, and new initial states will occur while existing  ones will disappear.

In the present paper we will address this drawback and present a {\em decentralized} design method to compute a distributed controller: each agent will compute its own local control law. For this computation, the agent will not need knowledge of the network graph or the initial states of all other agents. This will then enable `plug-and-play' operations on the network, since each agent will be able to automatically recompute its local gain whenever a new agent is added or removed.

In order to achieve this decentralized computation scheme we will write the original global cost functional as the sum of {\em local} LQ tracking cost functionals, each associated with one of the agents. The agents can not solve these optimal tracking problems explicitly because the reference signals depend on the future dynamics of the neigbours. However, using sampling, suboptimal local gains are obtained. This decentralized computation will not necessarily result in optimality of the global network behavior. We will however show that the resulting network will reach consensus.

The distributed LQ  control problem has attracted much attention in the past, see e.g. \cite{tamas2008,wei_ren2010,lunze_ecc_2013,Mosebach2014}. 
%
In \cite{tamas2008}, a suboptimal distributed controller for a global cost functional was developed to stabilize a network with general agent dynamics. 
A similar cost functional was also considered in \cite{GuaranteedLQR} for designing distributed controllers with guaranteed performance.
The distributed LQ control problem with general agent dynamics was also dealt with in  \cite{kristian2014} and  \cite{Inverse2015} by adopting an inverse optimal control approach. In \cite{SEMSARKAZEROONI20092205} a game theoretic approach was considered to obtain a suboptimal solution. Also, \cite{Jiao2018} considers a suboptimal version of this problem.
In \cite{Nguyen2015}, a suboptimal consensus controller design was developed by employing a hierarchical LQ control approach for an appropriately chosen global performance index, and a similar idea for constructing a particular cost functional was employed in \cite{nguyen_2017} to design a reduced order distributed controller.  
In \cite{Decoupling2017} a distributed optimal control method was adopted to decouple a class of linear multi-agent systems with state coupled nonlinear uncertainties.

The common feature of all work referred to above is that the computation of the control gains needs global information on the network. 
This disadvantage can be avoided by adopting adaptive control methods \cite{Ding2015} or by using reinforcement learning \cite{VAMVOUDAKIS20121598}, \cite{MODARES2016334}.
In  \cite{frank_structure_2015} and \cite{7862732}, it was shown that diffusive couplings are necessary for minimization of cost functionals of a particular form, involving the weighted squared synchronization error.

Below we list the contributions of the present paper.
\begin{enumerate}
\item
 We show that for agents with single integrator dynamics, in any distributed LQ cost functional the state weighting matrix must be equal to a weighted square of the Laplacian of the network graph.
 \item
 We give a solution to this general distributed LQ problem, and show that computation of the optimal protocol requires exact knowledge of the Laplacian and the initial state of the entire network.
 \item 
We represent the global cost functional as a sum of local LQ tracking cost functionals, one for each agent. Using sampling, suboptimal local gains are obtained.  Computation of these gains is completely decentralized. 
\item 
We show that these gains lead to a protocol that achieves consensus of the network.
\end{enumerate}

The outline of this paper is as follows. 
In Section \ref{sec_dis_LQ}, we derive a general formulation of the distributed LQ problem. 
In Section \ref{sec_centralized_gain_compu}, we show that computation of the optimal control laws requires complete knowledge of the network graph and the initial state of the entire network. 
In Section \ref{decentralized_gain_compu}, we propose a decentralized method to compute suboptimal (local) control laws. In order to do this, we need to apply ideas from linear quadratic tracking, and these are reviewed in Section \ref{sec_lqt_problem}.
Then, in Section \ref{sec_main_results}, we compute these local control laws, and show that the network reaches consensus if all agents apply their own local gain.
To illustrate the designed control protocol, a simulation example is provided in Section \ref{sec_simulation}.
Finally, in Section \ref{sec_conclusion}, we will give some concluding remarks.

\subsection*{Notation}
We denote by $\mathbb{R}$ the field of real numbers. The space of $n$-dimensional real vectors is denoted by $\mathbb{R}^n$.
The vector in $\mathbb{R}^N$ with all components equal to 1 is denoted by $\mathbf{1}_N$.
%
%
%
The identity matrix of dimension $n \times n$ is denoted by $I_n$.
%
%
For a symmetric matrix $P$, we write $P>0$ ($P \geq 0$) if $P$ is positive (semi-)definite.
We use $\text{diag}(a_1, a_2, \ldots, a_n)$ to denote the $n \times n$ diagonal matrix with $a_1, a_2, \ldots, a_n$ on its diagonal.
For a linear map $A: \mathcal{X} \to \mathcal{Y}$, the kernel and image of $A$ are denoted by $\ker (A) : = \{ x \in \mathcal{X} \mid A x =0 \}$ and ${\rm im}(A) := \{ Ax \mid x \in \mathcal{X} \}$, respectively.

	In this paper, a graph is denoted by $\mathcal{G} = (\mathcal{V},\mathcal{E})$ with $\mathcal{V} = \{ 1,  2, \ldots, N \}$ the node set and $\mathcal{E} \subset \mathcal{V}\times \mathcal{V}$ the edge set. 
	For $i,j \in \mathcal{V}$, an edge from node $i$ to $j$ is represented by $(i,j) \in \mathcal{E}$.
	The neighboring set of node $i$ is defined as $\mathcal{N}_i := \{ j \in \mathcal{V} \mid (i, j) \in \mathcal{E} \}$.
	The adjacency matrix  of $\mathcal{G}$ is equal to $A = [a_{ij}] \in \mathbb{R}^{N \times N}$, where $a_{ij} = 1$ if $(j,i) \in \mathcal{E}$ and $a_{ij} =0$ otherwise. 
	The degree matrix of $\mathcal{G}$ is given by $D = \text{diag}(d_1,d_2,\ldots,d_N)$ with $d_i = \sum_{j=1}^N a_{ij}$, and the Laplacian matrix is defined as $L := D -A$.
	A graph is called simple if $\mathcal{E}$ only contains edges $(i,j)$ with $i \neq j$, and it is 
	called undirected if $(i,j) \in \mathcal{E}$ implies that $(j,i) \in \mathcal{E}$. Obviously, a graph is undirected if and only if $L$ is symmetric.
	A simple undirected graph is called connected if for each pair of nodes $i$ and $j$ there exists a path from $i$ to $j$.
	Throughout this paper, it will be a standing assumption that the network graph is a connected simple undirected graph. 
	%
		%
	


\section{The general form of a distributed LQ cost functional}\label{sec_dis_LQ}
In this section we will show that in any distributed LQ cost functional, the state weighting matrix must be a weighted square of the Laplacian of the network graph. We will also give two important examples of distributed LQ cost functionals. 

We consider a network of agents described by scalar single integrator dynamics
\begin{equation}\label{one_agent}
	\dot{x}_i(t) = u_i(t),\quad x_i(0) = x_{i0},\quad i =1,2,\ldots, N,
\end{equation}
with $x_{i0} \in \mathbb{R}$ the initial state of agent $i$. By collecting the states and inputs of the individual agents into the vectors $x = (x_1, x_2, \ldots, x_N)^\top$ and $u = (u_1, u_2, \ldots, u_N)^\top$, \eqref{one_agent} can be written as
\begin{equation} \label{all_agents}
	\dot{x}(t) = u(t),\quad x(0) = x_0.
\end{equation}
A general class of LQ cost functionals are those of the form
\begin{equation} \label{general_LQ}
J(x_0,u) = \int_0^{\infty} x^\top(t) Qx(t) + u^\top(t) Ru(t) dt,
\end{equation}
where $Q \in \mathbb{R}^{N\times N}$, $R \in \mathbb{R}^{N\times N}$ and $Q \geq 0$ and $R >0$.

In the context of distributed LQ control we only allow distributed diffusive control laws that achieve consensus, i.e. the controlled trajectories converge to ${\rm im}(\mathbf{1}_N)$, the span of the vector of ones.  Thus the class of control laws over which we want to minimize \eqref{general_LQ} consists of those of  the form $u = -g Lx$, with  $L \in \mathbb{R}^{N \times N}$  the Laplacian of the network graph and where $g>0$, see e.g.  \cite{Olfati-Saber2004}. 

We will now show that for a cost functional \eqref{general_LQ} to make sense in this context, the weighting matrix $Q$ must be of the form $Q = LWL$ for some positive semi-definite matrix $W$.
\begin{lem} \label{Specialform}
$J(x_0,u) < \infty$ for all $x_0 \in \mathbb{R}^N$ and control laws of the form $u = -g Lx$  with $g >0$ only if there exists a positive semi-definite $ W \in \mathbb{R}^{N \times N}$ such that $Q = LWL$.
\end{lem} 
\begin{proof}
Write $Q = C^TC$ for some $C$. Now, let $\bar{x}(t)$ denote any nonzero state trajectory generated by the control law $u = - g Lx$ with $g >0$ and let $\bar{u}(t) = - g L\bar{x}(t)$. It is well known that this control law achieves consensus (see \cite{Olfati-Saber2004}) so we have $\bar{x}(t) \rightarrow c \mathbf{1}_N$ for some nonzero constant $c$. Now assume that the control law $u = - gLx$ gives finite cost, i.e. $J(x_0,\bar{u}) < \infty$.  This implies  $\int_0^{\infty} \bar{x}^\top(t)C^\top C\bar{x}(t) dt < \infty$ and hence $C\bar{x}(t) \rightarrow 0$. Thus we obtain $\mathbf{1}_N \in \ker(C)$, equivalently, $\ker(L) \subset \ker (C)$. 
We thus conclude that there exists a matrix $V$ such that $C = V L$ so the state weighting matrix $Q$ must be of the form $Q = L V^\top V L$ for some matrix $V$. This proves our claim.
\end{proof}
We have thus shown that, for a general LQ cost functional to make sense in the context of distributed diffusive control for multi-agent systems, it must necessarily be of the form 
\begin{equation} \label{general_LQ_new}
J(u,x_0) = \int_0^{\infty} x^\top(t) LWL x(t) + u^\top(t) Ru(t) dt,
\end{equation}
for some $W \geq 0$ and $R >0$.
The corresponding distributed LQ problem is to minimize, for the system \eqref{all_agents} with initial state $x_0$, the cost functional \eqref{general_LQ_new} over all control laws of the form $u = - g Lx$ with $g > 0$.

As an illustration, we will now provide two important special cases of LQ cost functionals.
The first one was studied before in \cite{Jiao2018} and \cite{wei_ren2010}: 
\begin{equation}\label{cost_i}
J(u, x_0) \!=\sum_{i=1}^{N} \int_{0}^{\infty}\!\!   \sum_{j\in \mathcal{N}_i}q (x_i(t)-x_j(t))^2 + r u_i^2(t) dt,
\end{equation}
where $q$ and $r$ are positive real numbers. Clearly, \eqref{cost_i} is equal to 
$
J(x_0,u) = \int_0^{\infty} x^\top(t) 2qL x(t) + r u^\top(t) u(t) dt.
$
Note that $2qL = L (2qL^{\dagger}) L$ with $L^{\dagger}$ the Moore-Penrose inverse of $L$ (which is indeed positive semi-definite). Thus this cost functional is of the form \eqref{general_LQ_new} with 
$W =2 qL^{\dagger}$ and $R = r I$. 

As a second example, consider
\begin{equation}\label{cost_r}
	J(x_0,u)=  \sum_{i=1}^{N}\int_{0}^{\infty}   q   \left( x_i(t) -  a_i(t)\right)^2 + r u_i^2(t) dt,
\end{equation}
with
\begin{equation}\label{r_i_t} 
	a_i(t) := \frac{1}{d_i +1} \big(x_i(t)  + \sum_{j \in \mathcal{N}_i}x_j(t) \big). 
\end{equation}
Here, $q$ and $r$ are positive weights, $d_i$ denotes the node degree of agent $i$ and $\mathcal{N}_i$ its set of neighbors.
The idea of the cost functional \eqref{cost_r} is to minimize the sum of the deviations between the state $x_i(t)$ and the average $a_i(t)$ of the states of its neighbors (including itself) and the control energy. In order to put this in the form \eqref{general_LQ_new}, define
\begin{equation}\label{matrix_G}
	G := (D + I_N)^{-1} (A + I_N) \in \mathbb{R}^{N\times N},
\end{equation}
where $D \in \mathbb{R}^{N\times N}$  is the degree matrix and $A \in \mathbb{R}^{N\times N}$  the adjacency matrix. Then clearly $a(t) = Gx(t)$, where $x =(x_1,x_2, \ldots, x_N)^\top$ and $a = (a_1,a_2, \ldots, a_N)^\top$. It is then easily seen that
\[
J(x_0,u)= \int_{0}^{\infty}\!\!\!\!q x\!^\top\!(t) (I_N -G)\!^\top\! (I_N -G) x(t) + ru\!^\top\!(t)u(t) dt.
\]
Since $(I_N -G)^\top(I_N -G) =L (D + I_N)^{-2}L$, we conclude that \eqref{cost_r} is a special case of 
\eqref{general_LQ_new} with  $W = q(D + I_N)^{-2}$ and  $R = r I_N$.

\section{Centralized Optimal Gain} \label{sec_centralized_gain_compu}
In this section we will briefly give a solution to the general distributed LQ problem with cost functional \eqref{general_LQ_new} as introduced  in Section 
\ref{sec_dis_LQ}, thus generalizing the result from \cite{wei_ren2010} to general distributed LQ cost functionals. We will show that, indeed, computation of the optimal protocol requires global information on the network graph and the initial state of the entire network.

Consider the cost functional \eqref{general_LQ_new} together with the dynamics \eqref{all_agents} with given initial state $x_0$. 
Since the admissible control laws are given by $u = -gLx$, the associated state trajectory is $x(t) = e^{-gL t} x_0$ and $u(t) = -gLx(t)$. Substituting this into the cost functional yields
\begin{equation} \label{function_k}
	J(g):= x_0^\top(\int_0^{\infty}\!\!\!\!\!\! e^{-g Lt} \left( LWL \!+\! g^2 LRL \right) e^{-gLt} dt)x_0
\end{equation}
Clearly, we need to minimize $J(g)$ over $g >0$.  
Substituting $g t = \tau$, we find 
\[
J(g):= x_0^\top \int_0^{\infty} e^{-\tau L} \left( \frac{1}{g} LWL + g LRL \right) e^{-\tau L} d\tau ~x_0.
\]
Define $X_0 := \int_0^{\infty} e^{-\tau L}LWL e^{-\tau L} d\tau$ and $Y_0 := \int_0^{\infty} e^{-\tau L}LRL e^{-\tau L} d\tau$. 
It turns out that both integrals indeed exist, and can be computed as particular solutions of the Lyapunov equations
\begin{subequations}
\begin{gather} 
-LX -XL + LWL =0\label{Lyapunov_1}\\
-LY -YL+ LRL =0\label{Lyapunov_2}
\end{gather}
\end{subequations}
Indeed, although $L$ is not Hurwitz, these equations do have positive semi-definite solutions $X$ and $Y$ and, in fact, $X_0$ is the unique positive semi-definite solution $X$ to \eqref{Lyapunov_1} with the property that ${\rm im}(\mathbf{1}_N) \subset \ker(X)$. Likewise $Y_0$ is the unique positive semi-definite solution $Y$ of \eqref{Lyapunov_2} with the property that ${\rm im}(\mathbf{1}_N)\subset \ker(Y)$ (see Proposition 1 in \cite{hiddejan}). It follows from \eqref{Lyapunov_2} that, in fact, $\ker(Y_0) = {\rm im}(\mathbf{1}_N)$.
Thus we see that $J(g) = \frac{1}{g} x^\top_0 X_0 x_0 + g x_0^\top Y_0 x_0$. 

In order to minimize $J(g)$ we distinguish three cases. (i) If $x_0 \in \ker(Y_0) = {\rm im}(\mathbf{1}_N)$ then we must have $x_0 \in \ker(X_0)$ as well, so $J(g) =0$ for all $\theta$ and every $g>0$ is optimal. (ii) If $x_0^\top Y_0 x_0 > 0$ and $x_0^\top X_0 x_0 =0$ then no optimal $g>0$ exists. (iii) If $x_0^\top Y_0 x_0 > 0$ and $x_0^\top X_0 x_0 > 0$ then an optimal $g >0$ exists and can be shown to be equal to
$
g^{\ast} = \left( \frac{x_0^\top X_0 x_0}{x_0^\top Y_0 x_0} \right) ^{\frac{1}{2}}.
$
It is clear that the computation of the optimal gain $g$ requires exact knowledge of the network graph in the form of the Laplacian $L$. Also, the optimal gain clearly depends on the global initial state of the network. 

\section{Towards Decentralized Computation}\label{decentralized_gain_compu}
In this section we will propose a new approach to compute `good' local gains that can be computed in a decentralized way. Instead of doing this for the general LQ cost functional \eqref{general_LQ_new}, we will zoom in on the particular case given by \eqref{cost_r}-\eqref{r_i_t}. 


%
In order to decentralize the computation, instead of minimizing the global cost functional \eqref{cost_r} for the multi-agent system \eqref{all_agents}, we write it as a sum of {\em local} cost functionals, one for each agent in the network. 

More specifically, the associated local cost functional for agent $i$ is given by
\begin{equation}\label{local_cost_r}
	J_{i} (u_i)= \int_{0}^{\infty} q \left( x_i(t) -  a_i(t )\right)^2 + r u_i^2(t) \  dt,
\end{equation}
where $a_i(t)$ is defined in \eqref{r_i_t}, for $ i =1,2,\ldots,N$.
This local cost functional penalizes the squared difference between the state of the $i$th agent and the average of the states of its neighboring agents (including itself), and the local control energy.
By minimizing \eqref{local_cost_r}, agent $i$ would make the difference between its own state and the average of the states of its neighbors (including itself) small. 
Note, however, that it is impossible for agent $i$ to minimize this local cost functional since the trajectory $a_i(t)$ for $t \in [0, \infty)$ associated with the neighboring agents is {\em not known}, so  also not available to the $i$-th agent.  Thus, 
because direct minimization of \eqref{local_cost_r} is impossible, as an alternative we will replace each of these local optimal control problems by {\em a sequence of linear quadratic tracking problems} that do turn out to be tractable.

More specifically, we choose a sampling period $T>0$, and introduce the following sampling procedure.
%
%
For each nonnegative integer $k$, at time $t= kT$ the $i$-th agent receives the sampled state value $x_j(kT)$ of its neighboring agents and takes the average of these, which is given by
\begin{equation}\label{r_tk}	
a_i(kT) = \frac{1}{d_i +1} \big(x_i(kT)  + \sum_{j \in \mathcal{N}_i}x_j(kT) \big).
\end{equation}
Then,  the $i$-th agent minimizes the cost functional
\begin{equation}\label{discounted_cost_r}
	J_{i,k} (u)\!=\!\! \int_{0}^{\infty}\!\!\!\!\!\! e^{-2\alpha t}\! \left( q \left( x_i(t) \!-\!  a_i(kT)\right)^2\! +\! r u_i^2(t) \right)\! dt.
\end{equation}
In fact, this is a discounted linear quadratic tracking problem with constant reference signal $a_i(kT)$ and discount factor $\alpha >0$. By solving this linear quadratic tracking problem, agent $i$ obtains an optimal control law over an infinite time interval.
However, agent $i$ applies this control law only on the time interval  $[kT, (k+1)T)$. 

Then, at time $t = (k+1)T$ the above procedure is repeated, i.e. agent $i$ receives the updated average $a_i((k+1)T)$, and subsequently solves the discounted tracking problem with cost functional $J_{i,k+1} (u)$ which involves the constant updated reference signal $a_i((k+1)T)$.
By performing this control design procedure sequentially at each sampling time $kT$, we then obtain a single control law for agent $i$ over the entire interval $[0, \infty)$.

	Based on this control design procedure for the individual agents, we will obtain a distributed control protocol for the entire multi-agent system, simply by letting all agents compute their own control law. 
	In the sequel we will analyze this protocol and show that it achieves consensus for the network:
\begin{defn}\label{def_consensus}
	A distributed control protocol  is said to achieve consensus for the network if $x_i(t) - x_j(t) \to 0$ as $t \to \infty$ for all initial states of agents $i$ and $j$, for all $i, j =1,2,\ldots,N$.
\end{defn}
%
%
	
	In order to obtain an explicit expression for the control protocol proposed above, we will study the linear quadratic tracking problem for a single linear system. This will be done in the next section.

%


\section{The Discounted LQ Tracking Problem}\label{sec_lqt_problem}
In this section, we will deal with the discounted linear quadratic tracking problem for a given linear system. The linear quadratic tracking problem has been studied before, see e.g. \cite{6787009}. 
Here, however, we will solve it by transforming it into a standard linear quadratic control problem.

Consider the continuous-time linear time-invariant system
\begin{equation}\label{sys_x}
	\dot{x}(t) = A x(t) + Bu(t),\quad x(0) = x_0,
\end{equation}
with $A \in \mathbb{R}^{n \times n}$ and $B \in \mathbb{R}^{n \times m}$, and where $x(t) \in \mathbb{R}^n$, $u (t) \in \mathbb{R}^m$ denote the state and the input, respectively.
We assume that the pair $(A, B)$ is stabilizable.
Given is also a contant reference signal $r_{\rm ref}(t) = r$ with $ r \in \mathbb{R}^n$.
Next,  we introduce a discounted quadratic cost functional given by
\begin{equation}\label{cost_x}
\!\!\!\!J(u) \!= \!\!\!\int_{0}^{\infty}\!\!\!\!\!\!\!  e^{-2\alpha t} [ {\left( x(t) \!  -\!  r\right)\!}^\top \!\!Q \left( x(t)\!   - \! r\right) 
 \!+\! {u\!}^\top\!(t) R u(t)  ] dt
\end{equation}
where $Q \in \mathbb{R}^{n\times n}$, $R \in \mathbb{R}^{m\times m}$ and $Q > 0$ and $R>0$ are given weight matrices and $\alpha >0$ is a discount factor \cite{6787009}.
%
%
The linear quadratic tracking problem is  to determine for every initial state $x_0$ a piecewise continuous input function $u(t)$ that minimizes the cost functional \eqref{cost_x}.

To solve this problem, we introduce the  variables 
\begin{equation}\label{new_variables}
	z(t) = e^{-\alpha t}x (t), \ z_r(t) = e^{-\alpha t}  r,\ v(t) = e^{-\alpha t} u(t),
\end{equation}
and denote ${\xi}(t) = (z^\top(t), z_r^\top(t))^\top$.
Then we obtain an auxiliary system in terms of $\xi$ and $v$, given by
\begin{equation*}\label{sys_z}
\dot{\xi}(t)
= 
A_e
\xi(t)
+ 
B_e v(t), \quad \xi_0= (x^\top_0, r^\top)^\top,
\end{equation*}
where $\xi_0 \in \mathbb{R}^{2n}$ is the initial state and
\begin{equation*}\label{barA_barB} 
A_e = 
\begin{pmatrix}
A-\alpha I_n & 0 \\
0 & -\alpha I_n
\end{pmatrix},\quad
B_e = 
\begin{pmatrix}
B \\
0
\end{pmatrix}.
\end{equation*}
In terms of the new variables $\xi$ and $v$, the cost functional  \eqref{cost_x} can be written as 
$	J(v) = \int_{0}^{\infty} 
	\xi^\top(t)
	Q_e
	\xi(t)
	+
	v ^\top (t)R v(t) \ dt ,$
where
$Q_e =
\left(\begin{matrix}
Q & -Q \\
-Q & Q
\end{matrix}\right) \in \mathbb{R}^{2n \times 2n}.$ 
%
The problem is now to find, for every initial state $\xi_0$, a piecewise continuous input function $v(t)$ that minimizes this cost functional.
%
This is a so-called a {\em free endpoint} standard LQ control problem, see \cite[pp. 218]{harry_book}.  
Since the pair $(A, B)$ is stabilizable, the pair $(A_e, B_e)$ is also stabilizable and hence the input function $v(t)$ that minimizes the cost functional $J(v)$ is generated by the feedback law
\begin{equation}\label{control_v}
	v(t) = -R^{-1} B_e^\top P_e^- \xi(t),
\end{equation}
where  $P_e^- \in \mathbb{R}^{2n \times 2n}$ is the smallest positive semi-definite solution of the Riccati equation
\begin{equation}\label{are}
A_e^\top P_e^- + P_e^- A_e -P_e^- B_e R^{-1}B_e^\top P_e^- + Q_e =0.
\end{equation}
Now,  partition
$P_e^- := 
\begin{pmatrix}
P_1 & P_{12} \\
P_{12}^\top & P_2
\end{pmatrix},$
where all blocks have dimension $n \times n$. Recalling \eqref{new_variables} and \eqref{control_v}, we then immediately find an expression for the input function $u(t)$ that minimizes
the cost functional \eqref{cost_x} for the  system \eqref{sys_x} and reference signal $r_{\rm ref}(t) = r$.
\begin{thm}\label{thm_lqt}

	The input function $u(t)$ that minimizes the cost functional \eqref{cost_x} is generated by the control law
	\begin{equation}\label{control_u}
		u(t) = K_1 x(t) + K_2  r,
	\end{equation}
where $K_1 = - R^{-1}B^\top P_1$ and $K_2 = - R^{-1}B^\top P_{12}$.
\end{thm}

The proof follows immediately from the above considerations. See also \cite{6787009}.
\begin{rem}
Let $e(t) := x(t) -r$ denote the tracking error. Because $Q >0$, the control law \eqref{control_u} only guarantees that $\bar{e}(t) := e^{-\alpha t}e(t)$ tends to zero as $t$ goes to infinity. Thus, the feedback law that minimizes the LQ tracking cost functional \eqref{cost_x} only guarantees
the actual tracking error $e(t)$ to be exponentionally bounded with growth rate $\alpha>0$. Note that $\alpha >0$ can be taken arbitrarily small.

It will be shown however that, for the multi-agent system case, the control design method established in this section will, nevertheless, lead to a protocol that achieves consensus.
\end{rem}

%

\section{Consensus Analysis}\label{sec_main_results}
In this section, we will show that, by adopting the  control design method for the multi-agent system \eqref{all_agents} as proposed in Section \ref{decentralized_gain_compu}, the resulting distributed control protocol achieves consensus for the entire network.

As already explained in Section \ref{decentralized_gain_compu}, we choose a sampling period $T>0$ and introduce a sampling procedure.
For each nonnegative integer $k$, at time $t= kT$ the $i$-th agent receives the sampled state value of its neighboring agents (including itself) and minimizes the cost functional \eqref{discounted_cost_r}, which  is a discounted linear quadratic tracking problem with constant reference signal $r_{\rm ref}(t) = a_i(kT)$ and discount factor $\alpha >0$.

According to the theory on the discounted LQ tracking problem described in Section \ref{sec_lqt_problem},  the local optimal control law for agent $i$ at time $t =kT$ over the whole time horizon $[0, \infty)$ is therefore of the form
\begin{equation}\label{local_control}
	u_{i,k}(t) = g_{i,k} x_i(t) + g'_{i,k} a_i(kT),
\end{equation}
in which the control gains $g_{i,k}$ and $g'_{i,k}$ can be computed explicitly by solving the Riccati equation \eqref{are} associated with the LQ tracking problem for agent $i$.
\begin{lem}\label{lem_kT_optimal}
	Consider,  at time $t =kT$, the $i$-th  agent of the multi-agent system \eqref{one_agent} with associated local cost functional \eqref{discounted_cost_r}.
	%
	%
	Denote 
	\begin{equation*}
	\bar{A} =\begin{pmatrix}
	-\alpha & 0 \\
	0 & -\alpha
	\end{pmatrix},\quad
	\bar{B} =
	\begin{pmatrix}
	1 \\0
	\end{pmatrix},\quad
	\bar{Q} =
	\begin{pmatrix}
	q & -q \\
	-q & q	
	\end{pmatrix}.
	\end{equation*}
	Let $\bar{P}:=
	\begin{pmatrix}
	p_1 & p_{12} \\
	p_{12} & p_2
	\end{pmatrix}
	$
	be the smallest positive semi-definite solution of the Riccati equation
	\begin{equation}\label{are_pi}
	\bar{A}^\top \bar{P}  + \bar{P}  \bar{A} - r^{-1} \bar{P} \bar{B} \bar{B}^\top \bar{P}  + \bar{Q} =0.
	\end{equation}
	Then the local control law \eqref{local_control} with $g_{i,k} := -r^{-1} p_1$ and $g'_{i,k} := -r^{-1} p_{12}$ minimizes the cost \eqref{discounted_cost_r} for agent $i$.
\end{lem}

\begin{proof}
	This follows immediately from Theorem \ref{thm_lqt}.
\end{proof}

Next, agent $i$ applies the control law \eqref{local_control} only on the time interval  $[kT, (k+1)T)$. 
Then, at time $t = (k+1)T$ the above procedure is repeated.

Since, for all $i = 1,2,\ldots, N$ and $k = 0,1,\ldots$, the matrices $\bar{A}$,  $\bar{B}$ and $\bar{Q}$ are independent of $i$ and $k$, the same holds for the gains $g_{i,k}$ and $g'_{i,k}$.
In the sequel, we will therefore drop the subscripts in the control gains $g_{i,k}$ and  $g'_{i,k}$  and denote them by $ g$ and $g'$, respectively.
Moreover, using \eqref{are_pi}, we compute $g = r^{-1}(\alpha - \sqrt{\alpha^2 + rq})$ and $g' = - g$.

By performing this procedure sequentially at each sampling time $kT$, we then obtain a single control law for agent $i$ over the entire interval $[0, \infty)$ as
\begin{equation}\label{local_control_kT}
	u_{i,k}(t) = g x_i(t) -g  a_i(kT), \quad t \in [kT, (k+1)T),
\end{equation}
where $g = r^{-1}(\alpha - \sqrt{\alpha^2 + rq}) <0$.

Recall that $a(t) = Gx(t)$, with $G$ given by \eqref{matrix_G}, and that $a(kT) = G x(kT)$.  
Therefore, the local control laws for the individual agents lead to a distributed  control protocol
\begin{equation}\label{protocol}
u_{k}(t)  = g x(t) - g G x(kT),  \quad t \in [kT, (k+1)T).
\end{equation}
Now, by applying the  protocol \eqref{protocol} to the multi-agent system \eqref{one_agent},  we  find that the controlled network is represented by
\begin{equation}\label{network_compact}
\dot{x}(t) = g x(t) - g G x(kT),  \quad t \in [kT, (k+1)T).
\end{equation}
In the remainder of this section, we will analyze this representation, and show that consensus is achieved, i.e. for each initial state $x(0) = x_0$ we have $x_i(t) - x_j(t) \to 0$ as $t$ tends to infinity.

In order to do this, note that the solution of \eqref{network_compact} with initial state $x(0) = x_0$ is given by
\begin{equation}  \label{expl}
x(t) = e^{g(t-kT)} x(kT) - \int_{kT}^{t} e^{g(t-\tau)} g  G \ x(kT) \ d\tau,
\end{equation}
for $t \in [kT, (k+1)T)$, $k =0,1,2, \ldots$. Obviously, for each initial state $x_0$, the corresponding solution $x(t)$ is continuous.
From \eqref{expl} we see that the sequence of network states $x(kT)$ evaluated at the discrete time instances $kT$, $k=0,1, \ldots$ satisfies the difference equation
\begin{equation}\label{Gamma_1}
	x((k+1)T) =\Gamma x(kT),
\end{equation}
 $\Gamma = e^{g T} I_N - ( e^{g T} - 1) G \in \mathbb{R}^{N \times N}$. 

Clearly, the network reaches consensus if and only if for each $x_0$, $x_i(kT) - x_j(kT) \to 0$ as $t$ tends to infinity.

We proceed with analyzing the eigenvalues of $G$.
\begin{lem}\label{lem_G}
	The  matrix $G$  has an eigenvalue $1$ with algebraic multiplicity equal to one and associated eigenvector $\mathbf{1}_N$.
	The remaining eigenvalues of $G$ are all real and have absolute value strictly less than $1$.
\end{lem}

\begin{proof}
	Since ${L} = {D} - {A}$,  we have $G = I_N - ({D} + I_N)^{-1}  {L}.$
	%
	Hence we have $\tilde{D}^{\frac{1}{2}}G \tilde{D}^{-\frac{1}{2}} =  I_N - \tilde{D}^{-\frac{1}{2}} {L}\tilde{D}^{-\frac{1}{2}}$ where $\tilde{D}=D+I_N$. 
	Note that the right hand side is symmetric and hence has only real eigenvalues. 	%
	Thus, by matrix similarity, $G$ also has only real eigenvalues.

	Next, we show that $G$ has  a simple eigenvalue $1$ with associated eigenvector  $\mathbf{1}_N$.
	First note  that
	\begin{equation}
	G \mathbf{1}_N = (I_N - ({D} + I_N)^{-1}  {L}) \mathbf{1}_N = \mathbf{1}_N.
	\end{equation}
	Hence, indeed, $1$ is an eigenvalue of $G$ with eigenvector $\mathbf{1}_N$. Since $G$ is similar to a symmetric matrix, it is diagonalizable, so the algebraic multiplicity of its eigenvalue $1$ must be equal to its geometric multiplicity. Suppose now that 1 is not a simple eigenvalue. Then there must exist a second eigenvector, say $v$, which is linearly independent of $\mathbf{1}_N$. This implies $Gv = v$. Then $Lv = 0$, so $v$ must be a multiple of $\mathbf{1}_N$. This is a contradiction. We conclude that the eigenvalue $1$ is indeed simple.

Finally, it follows from Gershgorin's Theorem  \cite{Horn1990} that every eigenvalue $\lambda$ of $G$  satisfies $-1 < \lambda \leq  1$.
\end{proof}

Before we give the main result of this paper, we first review the following proposition.

\begin{prop}\label{prop_output_stable}
	Consider the discrete-time system
\begin{equation*}
		x(k+1) = A x(k),\quad x(0)=x(0),\quad y(k) = Cx(k)
\end{equation*}
	with $A \in \mathbb{R}^{n \times n}$ and $C \in \mathbb{R}^{p \times n}$, where $x(k) \in \mathbb{R}^n$ is the state, $x_0$ is the initial state and $y(k) \in \mathbb{R}^p$ is the output.
	Then, $y(k) \to 0$ as $k \to \infty$ for all  initial states $x_0$ if and only if 
$X_+(A) \subset \ker (C).$ Here, $X_+(A)$ is the unstable subspace, i.e., the sum of the generalized eigenspaces of $A$ associated with its eigenvalues in $\{\lambda \in \mathbb{C} \mid |\lambda| \geq 1 \}$.
\end{prop}
\begin{proof}
	A proof can be given by generalizing the results \cite[pp. 99]{harry_book} to the discrete time case.
\end{proof}

We are now ready to present the main result of this paper.

\begin{thm}\label{main_thm}
	Consider the multi-agent system \eqref{one_agent}. Let $T >0$ be a sampling period, $\alpha >0$ a discount factor, and let $q,r > 0$ be given weights. 
	Let $\bar{P}$ be the smallest positive semi-definite solution of the Riccati equation \eqref{are_pi} and partition 
	$
	\bar{P}:=
	\begin{pmatrix}
	p_1 & p_{12} \\
	p_{12} & p_2
	\end{pmatrix}.
	$
	Then the distributed control protocol  \eqref{protocol} with $g = -r^{-1} p_1$ and $g' = -r^{-1} p_{12}$ achieves consensus for the controlled network \eqref{network_compact}.
\end{thm}
\begin{proof}
	The network reaches consensus if and only if $Lx(kT) \to 0$ as $k \to \infty$.
	Since $\ker(L) = {\rm im} (\mathbf{1}_N)$, it then follows from Proposition \ref{prop_output_stable} that consensus is achieved if and only if $X_+(\Gamma) \subset \ker(L)$, equivalently, the sum of the generalized eigenspaces of $\Gamma$ corresponding to the eigenvalues $\lambda$ with $\mid\lambda\mid \geq 1$ is equal to  ${\rm im} (\mathbf{1}_N)$.
	
	Indeeed, we will show that all eigenvalues $\lambda$ of $\Gamma$ are real and satisfy $-1 < \lambda \leq 1$, and $\lambda =1$ is a simple eigenvalue with associated eigenvector $\mathbf{1}_N$.

	Recall that $\Gamma = e^{g T} I_N - ( e^{g T} - 1) G$. Hence, $\mu$ is an eigenvalue of $\Gamma$ if and only if $\mu = e^{g T}  - \lambda ( e^{g T} - 1)$ where $\lambda$ is an eigenvalue of $G$. 
It was shown in Lemma \ref{lem_G} that all eigenvalues $\lambda$ of $G$ are real and satisfy $-1 < \lambda \leq 1$ and, moreover, $\lambda =1$ is a simple eigenvalue. Using the fact that $g < 0$ we thus obtain that the eigenvalues $\mu$ of $\Gamma$ satisfy $-1 < \mu \leq 1$ and $\mu =1 $ is a simple eigenvalue of $\Gamma$.

Finally, we will show $\mu =1$ has eigenvector $\mathbf{1}_N$. Indeed, this follows from
$
\Gamma \mathbf{1}_N= (e^{g T} I_N - ( e^{g T} - 1) G) \mathbf{1}_N = \mathbf{1}_N.$
 	This completes the proof.
\end{proof}
\begin{rem}
By analyzing the eigenvalues $\mu$ of $\Gamma$ satisying $-1 < \mu <1$, it can be seen that, for given $\alpha$, the convergence rate of the difference equation \eqref{Gamma_1} increases with increasing sampling period $T$. The total time it takes to reach a disagreement smaller than a given tolerance is then the product of the number of iterations in \eqref{Gamma_1} and this sampling period. It might therefore be more advantageous to use a smaller sampling period with a larger number of required iterations, but yet leading to a smaller total time.
In other words, the choice of sampling period is a trade-off between the total time required to obtain an acceptable disagreement, and the number of iterations in \eqref{Gamma_1}. 
\end{rem}

\section{Simulation}\label{sec_simulation}
%
Consider a network of six agents with single integrator dynamics
$\dot{x}_i(t) = u_i(t),i =1,2,\ldots,6,$
where the initial states are $x_{10} = 1$, $x_{20} = 2$, $x_{30} = -1$, $x_{40} = -2$, $x_{50} = 1$ and $x_{60} = 3$.
We assume that the communication among these agents is represented by an undirected circle graph with six nodes.
First, we take the sampling period to be equal $T=10$. 
On the time interval $t \in [kT, (k+1)T)$, $k =0,1, \ldots$, we consider the local cost functional \eqref{discounted_cost_r} for agent $i$.
We choose the weights to be $q =2$, $r=1$ and the discount factor $\alpha = 0.01$.
We adopt the control design proposed in Theorem \ref{main_thm} and compute the smallest positive semi-definite of the Riccati equation
$A^\top {P}  +{P}  {A} -  r^{-1}  {P} {B} {B}^\top {P}  + {Q} =0$
with
\begin{equation*} 
	A = 
	\begin{pmatrix}
		-0.01 & 0 \\
		0 & -0.01
	\end{pmatrix},\
	B = 
	\begin{pmatrix}
		1 \\
		0
	\end{pmatrix},\
	Q = 
	\begin{pmatrix}
		2 & -2 \\
		-2 & 2
	\end{pmatrix}.
\end{equation*}
\noindent This Riccati equation has a unique positive semi-definite solution which is given by
\begin{equation*} 
	P = 
	\begin{pmatrix}
    1.4042  & -1.4042 \\
-1.4042  &  1.4042		
	\end{pmatrix}. 
\end{equation*}
Thus we find the control gains 
$	g =    -1.4042$ and
$	g' =     1.4042.$
Subsequently, the local control law for agent $i$ is given by
$
		u_{i,k}(t) =  -1.4042 x_i(t) +  1.4042 a_i(kT)
$
for $t \in [kT, (k+1)T)$ and $i = 1,2,3$ and $k = 0,1,\ldots$.

In Figure \ref{consensus1} we have plotted the controlled trajectories of the individual agents. It can be seen that  the protocol resulting from the local control laws indeed achieves consensus.
\begin{figure}[t!]
	\centering
	\includegraphics[height=5cm]{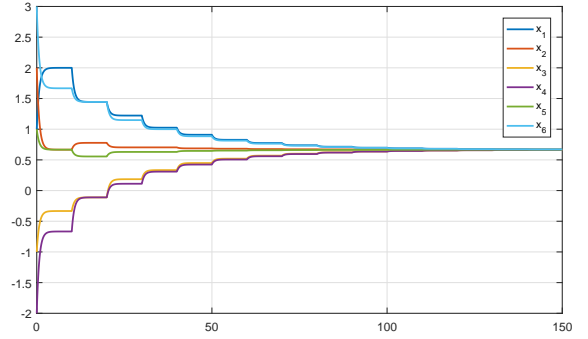}
	\caption{Plot of the states of the controlled network with $T =10$} \label{consensus1}
\end{figure}
%
The results of a second simulation, this time with sampling period $T =0.1$, are plotted in Figure \ref{consensus2}.
\begin{figure}[t!]
	\centering
	\includegraphics[height=5cm]{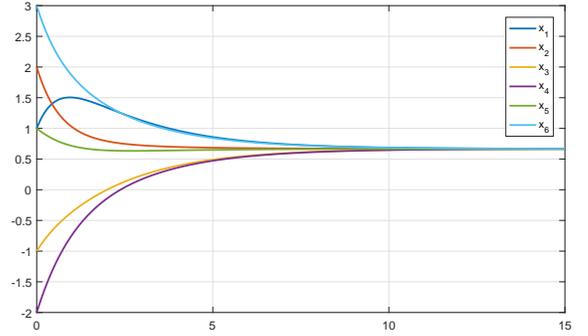}
	\caption{Plot of the states of the controlled network with $T =0.1$} \label{consensus2}
\end{figure}

By comparing  Figure~\ref{consensus1} and \ref{consensus2}, it can be seen that the network reaches consensus faster by taking a smaller sampling period.

\balance

\section{Conclusion}\label{sec_conclusion}
We have studied the distributed linear quadratic control problem for a network of agents with single integrator dynamics.
We have shown that the computation of control gains that minimize global cost functionals need global information, in particular the initial states of all agents and the Laplacian matrix. 
We have also shown that this drawback can be overcome  by transforming the global cost functional into discounted local cost functionals and assigning each of these to an associated agent.
In such a way, each agent computes its own control gain, using sampled information of its neighboring agents.
Finally, we have shown that the resulting control protocol achieves consensus for the network.




%
%
%




{\footnotesize
\bibliographystyle{ieee}
\bibliography{local_costs}  }

\end{document}